\documentclass[12pt]{article}

\usepackage{amsmath,amsfonts, amssymb}

\begin{document}
\bibliographystyle{plain}
\pagenumbering{arabic}
\raggedbottom

\newtheorem{theorem}{Theorem}[section]
\newtheorem{lemma}[theorem]{Lemma}
\newtheorem{proposition}[theorem]{Proposition}
\newtheorem{definition}[theorem]{Definition}

\setlength{\parskip}{\parsep}
\setlength{\parindent}{0pt}
\setcounter{tocdepth}{1}

\def \outlineby #1#2#3{\vbox{\hrule\hbox{\vrule\kern #1%
\vbox{\kern #2 #3\kern #2}\kern #1\vrule}\hrule}}%
\def \endbox {\outlineby{4pt}{4pt}{}}%

\newenvironment{proof}
{\noindent{\bf Proof\ }}{{\hfill \endbox
}\par\vskip2\parsep}

\hfuzz8pt

\newcommand{\cl}[1]{{\mathcal{C}}_{#1}}
\newcommand{\cov}{\rm Cov}
\newcommand{\der}[3]{\ra{#1}{#2}_{#3}}
\newcommand{\var}{\rm Var\;}
\newcommand{\tends}{\rightarrow \infty}
\newcommand{\ep}{{\mathbb {E}}}
\newcommand{\pr}{{\mathbb {P}}}
\newcommand{\re}{{\mathcal{R}}}
\newcommand{\rt}{\widetilde{\rho}}
\newcommand{\wt}[1]{\widetilde{#1}}
\newcommand{\bds}{\begin{displaystyle}}
\newcommand{\eds}{\end{displaystyle}}
\newcommand{\ra}[2]{{#1}^{(#2)}}
\newcommand{\rat}[1]{\ra{#1}{\tau}}
\newcommand{\sif}[2]{\Sigma_{#1}^{#2}}
\newcommand{\1}{\leavevmode\hbox{\rm \normalsize1\kern-0.26em I}}
\newcommand{\fish}{J_{\rm st}}

\title{Information inequalities and a dependent Central Limit Theorem}
\author{Oliver Johnson}
\date{September 5th 2001}
\maketitle

\makeatletter
\begin{abstract} 
We adapt arguments concerning information-theoretic convergence in the 
Central Limit Theorem to the case of
dependent random variables under Rosenblatt mixing conditions.
The key is to work with random variables perturbed
by the addition of a normal random variable, 
giving us good control of the joint density and the mixing coefficient. 
We strengthen results of
Takano and of Carlen and Soffer to provide entropy-theoretic, not weak
convergence. 
\renewcommand{\thefootnote}{}
\footnote{{\bf Address:} O.T.Johnson, Statistical Laboratory, 
Centre for Mathematical Sciences, Wilberforce Road, Cambridge, CB3 0WB,
UK.
Contact Email: {\tt otj1000@cam.ac.uk}. }
\renewcommand{\thefootnote}{\arabic{footnote}}
\setcounter{footnote}{0}
\makeatother
\end{abstract}

\section{Introduction and notation}
Under a variance constraint, entropy  is maximised by the Gaussian. It
is natural to consider whether  entropy converges to this maximum in
the  Central  Limit  Theorem  regime.   This  is  a  strong  sense  of
convergence,   and  is   discussed  by   Brown   \cite{brown},  Barron
\cite{barron} and Johnson \cite{johnson}.  These papers only deal with
the   case   of  independent   random   variables,  \cite{brown}   and
\cite{barron} in  the case  of identically distributed  variables, and
\cite{johnson} for non-identical variables satisfying a Lindeberg-like
condition.  This  paper extends  these techniques to  weakly dependent
random variables.

Takano   \cite{takano},  \cite{takano2}   considers  the   entropy  of
convolutions of dependent random variables, though he imposes a strong
$\delta_4$-mixing  condition  (see Definition  \ref{def:takanocoeff}).
Carlen and Soffer \cite{carlen}  also use entropy-theoretic methods in
the dependent  case, though the  conditions which they impose  are not
transparent. Takano, in common with  Carlen and Soffer, does not prove
convergence  in  relative  entropy  of  the full  sequence  of  random
variables,  but  rather convergence  of  the  `rooms' (in  Bernstein's
terminology),   equivalent  to  weak   convergence  of   the  original
variables.   Our   conclusion  is  stronger.   In   a  previous  paper
\cite{johnson8},   we    used   similar   techniques    to   establish
entropy-theoretic convergence for  FKG systems, which whilst providing
a  natural  physical  model,  restrict  us to  the  case  of  positive
correlation.

We will consider a doubly infinite stationary collection of random variables
$ 
\ldots, X_{-1}, X_0, X_1, X_2, \ldots$, with mean zero and finite variance. 
We write $v_n$ for $\var(\sum_{i=1}^n X_i)$ and $U_n = (\sum_{i=1}^n X_i)/
\sqrt{n}$. We will consider perturbed random variables 
$\rat{V}_n = (\sum_{i=1}^n X_i + \rat{Z}_i)/\sqrt{n} \sim U_n + \rat{Z}$,
for $\rat{Z}_i$
a sequence of $N(0,\tau)$ independent of $X_i$ and each other. 
In general, $\ra{Z}{s}$ will be a $N(0,s)$.  
If the limit $\sum_{j=-\infty}^{\infty} \cov(X_0, X_j)$ exists then we 
denote it by $v$.
\begin{definition}
Given  two  random variables $S,T$, 
the  $\alpha$-mixing coefficient is defined to be:
$$   \alpha(S,T)  =   \sup_{A,B}
\left|  \pr ((S \in A)  \cap  (T \in B))  -  \pr(S \in A) \pr(T \in B)  
\right|.$$  
If $\sif{a}{b}$ is the $\sigma$-field  generated by $X_a, X_{a+1}, \ldots
, X_b$ (where $a$ or $b$ can be infinite), then for each $t$, define:
$$ \alpha(t) = \sup \left\{ \alpha( S,T): S \in \sif{-\infty}{0}, 
T \in \sif{t}{\infty} \right\},$$
and define the process to be $\alpha$-mixing if $\alpha(t) \rightarrow 0$ as
$t \tends$. \end{definition}
See Bradley \cite{bradley} for a discussion of the 
properties and alternative definitions of mixing coefficients. 
Note that $\alpha$-mixing is sometimes referred to as strong 
mixing, and is implied by uniform mixing (control
of $|P(A | B) - P(A)|$, equivalent to the Doeblin condition for Markov 
chains). All $m$-dependent processes are $\alpha$-mixing,
as well as any stationary, real aperiodic Harris chain (which includes
every finite state irreducible aperiodic Markov chain).
\begin{definition}
For a random variable $U$ with smooth density $p$, we consider the score
function $\rho(u) = p'(u)/p(u)$, the Fisher information $J(U) = \ep \rho^2(U)$,
and the standardised Fisher information $\fish(U) = \sigma^2_U J(U) -1$.
\end{definition}
We continue the technique used to prove convergence in relative entropy
first developed by Barron \cite{barron}, and later adapted to the 
non-identical case by Johnson \cite{johnson}. That is, we use de Bruijn's
identity:
\begin{lemma} \label{lem:debruijn}
If $U$ is a random variable with density $f$ and variance 1,
and $\rat{Z}$ is a sequence of normals independent of $U$, then the 
relative entropy distance $D$ between $f$ and the standard Gaussian density
$\phi$ is given by:
$$ D(f \| \phi) = \frac{1}{2}
\int_{0}^{\infty} \left( J(U + \rat{Z}) - \frac{1}{1+\tau} \right) d\tau . $$
\end{lemma}
Our main theorems concerning strong mixing variables are as follows:
\begin{theorem} \label{thm:strongconv}
Consider a stationary collection of
random variables $X_i$, with finite $(2+\delta)$th moment. If
$ \sum_{j=1}^{\infty} \alpha(j)^{\delta/(2+\delta)} < \infty,$ then
for any $\tau >0$:
$$\lim_{n \tends} \fish(\rat{V}_n) \rightarrow 0.$$ \end{theorem}
Note that the condition on the $\alpha(j)$ implies that $v_n/n \rightarrow v
< \infty$ (see Lemma \ref{lem:ibragimov}). In the next theorem, we have to 
distinguish two cases, where $v=0$ and where $v>0$. For example, if 
$Y_j$ are IID, and $X_j = Y_j - Y_{j+1}$ then $U_n = (Y_1 - Y_{n+1})/\sqrt{n}
\rightarrow \delta_0$. However, since we make a normal 
perturbation, we know that $\fish ( \rat{V}_n) = (v_n/n + \tau) 
J( \rat{V}_n) - 1 \leq (v_n/n + \tau) 
J( \rat{Z}) - 1 = v_n/n \tau$, so the case $v=0$ automatically works in
Theorem \ref{thm:strongconv}.

We can provide a corresponding result for convergence in relative entropy,
with some extra conditions:
\begin{theorem} \label{thm:strongentconv}
Consider a stationary collection of
random variables $X_i$, with finite $(2+\delta)$th moment. 
If 
\begin{enumerate}
\item{$ \sum_{j=1}^{\infty} \alpha(j)^{\delta/(2+\delta)} < \infty$}
\item{$v = \sum_{j=-\infty}^{\infty} \cov(X_0,X_j) > 0$}
\item{If $\begin{displaystyle} f_N(\tau) = \sup_{n \geq N} \left( 
\frac{n\fish(\rat{V}_n)}{v_n +n\tau} \right)
\end{displaystyle}
$, for some $N$, $\int f_N(\tau) d\tau < \infty$}
\end{enumerate}
then writing $g_n$ 
for the density of $(\sum_{i=1}^n X_i)/\sqrt{v_n}$ then:
$$\lim_{n \tends} D(g_n \| \phi) \rightarrow 0.$$ \end{theorem}
\begin{proof} Follows from Theorem \ref{thm:strongconv} by a dominated
convergence argument using de Bruijn's identity, Lemma
\ref{lem:debruijn}. \end{proof}

Note that convergence in relative entropy is a strong result and implies 
convergence in $L^1$ and hence weak convergence of the original variables.

Convergence  of   Fisher  information,  Theorem  \ref{thm:strongconv},
is actually implied by 
Ibragimov's \cite{ibragimov} classical  weak convergence
result. This follows  since the density of $\rat{V}_n$ (and its derivative)
can  be  expressed as
expectations  of  a continuous  bounded  function  of $U_n$.   Shimizu
\cite{shimizu}  discusses  this technique,  which  can  only work  for
random variables perturbed by a  normal. We hope our method may be
extended  to  the general  case,  since  results  such as  Proposition
\ref{prop:fishdecomp} do not  need the random variables to  be in this
smoothed form. For example in the independent case, we show in a 
forthcoming paper that $\fish(U_n) \rightarrow 0$, if $J(U_m)$ is finite
for some $m$, and if $U_{2^k}$ is unimodal for infinitely many $k$
(no normal perturbation is necessary). In any case, we feel there is 
independent interest in seeing why the normal distribution is the limit of 
convolutions, as the score function becomes closer to the linear case
which characterises the Gaussian.
\section{Fisher Information and convolution}
\begin{definition}
For random variables $X$, $Y$ with score functions $\rho_X$, $\rho_Y$, for
any $\beta$, we define $\widetilde{\rho}$ for the score function of
$\sqrt{\beta} X + \sqrt{1-\beta} Y$ and then:
$$ \Delta(X, Y, \beta) =
\ep \left( \sqrt{\beta} \rho_X(X) + \sqrt{1-\beta}
\rho_Y(Y) - \widetilde{\rho}\left( \sqrt{\beta} X + \sqrt{1-\beta} Y \right) 
\right)^2.$$
\end{definition}
Firstly, we provide a theorem which tells us how Fisher information changes
on the addition of two random variables which are nearly independent.
\begin{theorem} \label{thm:fishsub}
Let $S$ and $T$ be random variables, with $\max(\var S, \var T) \leq K\tau$.
Define $X = S + \rat{Z}_S$ and 
$Y = T + \rat{Z}_T$ (for $\rat{Z}_S$ and $\rat{Z}_T$
 normal $N(0,\tau)$ independent of
$S$, $T$ and each other), with score functions $\rho_X$ and $\rho_Y$.
There exists a constant $C= C(K, \tau, \epsilon)$ such that:
$$ \beta J(X) + (1-\beta) J(Y) - J \left( \sqrt{\beta} X + 
\sqrt{1-\beta} Y \right) + C \alpha(S,T)^{1/3-\epsilon} 
\geq  \Delta(X, Y, \beta).$$
\end{theorem}

If $S,T$ have bounded $k$th moment, we can replace $1/3$ by
$k/(k+4)$.
The proof requires some involved analysis, and is deferred 
to Section \ref{sec:mainproof}. 

In comparison, Takano \cite{takano}, \cite{takano2} produces bounds which 
depend on $\delta_4(S,T)$, where:
\begin{definition} \label{def:takanocoeff}
For random variables $S,T$ with joint density $p_{S,T}(s,t)$ and marginal
densities $p_S(s)$ and $p_T(t)$, define the $\delta_n$ coefficient to be:
$$ \delta_n(S,T)
= \left( \int p_S(s) p_T(t) \left| \frac{p_{S,T}(s,t)}{p_S(s) p_T(t)} -1
\right|^n ds dt \right)^{1/n}.$$ \end{definition}
In the case where $S,T$ have a continuous joint density, it
is clear that Takano's condition is more restrictive, and lies between two
more standard measures of dependence:
$$ 4\alpha(S,T) \leq \delta_4(S,T) \leq \delta_\infty(S,T)
= \psi(S,T) = 
\sup_{A,B} \left| \frac{P(A \cap B)}{P(A)P(B)} - 1 \right|.$$
(as before see Bradley \cite{bradley} for a discussion of different mixing
conditions).

Another use of the smoothing of the variables allows us to control the 
mixing coefficients themselves:
\begin{theorem} \label{thm:contmix}
For $S$ and $T$, define  $X = S+ \rat{Z}_S$ and  $Y = T+ \rat{Z}_T$,
where $\max(\var S, \var T) \leq K\tau$. 
If $Z$ has variance
$\epsilon$, then there exists a function $f_K$ such that
$$ \alpha(X+Z, Y) \leq \alpha(X, Y) + f_K(\epsilon),$$
where $f_K(\epsilon) \rightarrow 0$ as $\epsilon \rightarrow 0$. \end{theorem}
\begin{proof} See Section \ref{sec:contmix}. \end{proof}

To complete our analysis, we need lower bounds on the term
$ \Delta(X, Y, \beta)$. For independent
$X$, $Y$ it equals zero exactly when $\rho_X$ and $\rho_Y$ 
are linear, and if it is small then $\rho_X$ and $\rho_Y$ are close
to linear. Indeed, in \cite{johnson} we make two definitions:
\begin{definition} For a function $\psi$, define the class of random variables
$X$ with variance $v_X$ such that:
$$ {{\cal{C}}}_\psi = \{ X: \ep X^2 \1(|X| \geq R\sqrt{v_X} ) \leq v_X \psi(R)
\} .$$
Further, define a semi-norm $\| \; \|_{\Theta}$ on functions via:
$$ \| f \|_{\Theta}^2 = \inf_{a,b} \; \; 
\ep \left( f(\ra{Z}{\tau/2}) - a \ra{Z}{\tau/2} -b
\right)^2.$$ \end{definition}
Combining results from previous papers we obtain:
\begin{proposition} \label{prop:deltadom}
For $S$ and $T$ with $\max(\var S,\var T) \leq K \tau$, define 
$X = S+ \rat{Z}_S$, $Y= T+ \rat{Z}_T$. 
For any $\psi$, $\delta >0$, there exists a function $\nu = \nu_{\psi,\delta,K,
\tau}$,
with $\nu(\epsilon) \rightarrow 0$ as $\epsilon \rightarrow 0$, such that
if $X,Y \in {\cal{C}}_\psi$, and 
$\beta \in (\delta, 1- \delta)$ then
$$ \fish(X) \leq \nu( \Delta(X, Y, \beta)).$$
\end{proposition}
\begin{proof} 
We reproduce the proof of Lemma 3.1 of Johnson and Suhov \cite{johnson3},
which implies $p(x,y) \geq (\exp(-4K)/4) \phi_{\tau/2}(x) \phi_{\tau/2}(y)$. 
This follows since by Chebyshev
$\int \1(s^2 + t^2 \leq 4K\tau) dF_{S,T}(s,t) \geq 1/2$, and
since $(x-s)^2 \leq 2x^2 + 2s^2$:
\begin{eqnarray*}
p(x,y) & = &
\int \phi_{\tau}(x-s) \phi_{\tau}(y-t) dF_{S,T}(s,t) \\ 
& \geq & \frac{1}{2} \min\{ \phi_{\tau}(x-s) \phi_{\tau}(y-t) 
 : s^2 + t^2 \leq 4K\tau \} \\
& = & \frac{\phi_{\tau/2}(x) \phi_{\tau/2}(y)}{4}  
\exp \left( \min_{s^2 + t^2 \leq 4K\tau}
\left\{ \frac{ -s^2 - t^2}{\tau} \right\} \right) \\
& \geq & \frac{1}{4} \exp( -4K) \phi_{\tau/2}(x) \phi_{\tau/2}(y)
\end{eqnarray*}
Hence writing $h(x,y) = \sqrt{\beta} \rho_X(x) + \sqrt{1-\beta}
\rho_Y(y) - \widetilde{\rho} \left( \sqrt{\beta} x + \sqrt{1-\beta} y 
\right)$, then:
\begin{eqnarray*}
\Delta(X,Y,\beta)
& = & \int p(x,y) h(x,y)^2 dx dy \\
& \geq & \frac{\exp(-8K)}{16} \int \phi_{\tau/2}(x) \phi_{\tau/2}(y)
h(x,y)^2 dx dy \\
& \geq & \frac{\beta(1-\beta) \exp(-8K)}{32} \left(
\| \rho_X \|_{\Theta}^2 + \| \rho_Y \|_{\Theta}^2 \right),
\end{eqnarray*}
by Proposition 3.2 of Johnson \cite{johnson}. 
The crucial result of \cite{johnson} implies that
for fixed $\psi$, if the sequence $X_n \in {{\cal{C}}}_\psi$ have 
score functions $ \rho_n$, then $\| \rho_n \|_{\Theta} \rightarrow 0$
implies that $\fish(X_n) \rightarrow 0$. 
\end{proof}

We therefore concentrate on random processes such that the sums
$(X_1 + X_2 + \ldots X_m)$ have uniformly decaying tails:
\begin{lemma}[Ibragimov, \cite{ibragimov}] \label{lem:ibragimov}
If $\{ X_j \}$ are stationary
with $\ep |X|^{2+\delta} < \infty$ for some $\delta >0$
and
$ \sum_{j=1}^{\infty} \alpha(j)^{\delta/\delta+2} < \infty,$
then
\begin{enumerate}
\item{$(X_1 + \ldots X_m)$ belong to some class ${{\cal{C}}}_\psi$, uniformly in $m$.}
\item{$v_n/n \rightarrow v = 
\sum_{j=-\infty}^{\infty} \cov(X_0, X_j) < \infty$.}
\end{enumerate}
\end{lemma}

We are able to complete the proof of the CLT, under strong mixing conditions.

\begin{proof}{\bf of Theorem \ref{thm:strongconv}}
Combining Theorems \ref{thm:fishsub} and \ref{thm:contmix}, and 
defining $\rat{\wt{V}}_n = (\sum_{i=m+1}^n X_i + \rat{Z}_i)/\sqrt{n}$, we
obtain that
for $m \geq n$,
$$ \fish(\rat{V}_{m+n})
\leq \frac{m}{m+n} \fish( \rat{V}_m) + \frac{n}{m+n} \fish( \rat{V}_n) +
c(m) - \Delta \left( \rat{V}_m, \rat{\wt{V}}_n, \frac{m}{m+n} \right),$$
where $c(m) \rightarrow 0$ as $m \rightarrow \infty$. We show this using the 
idea of `rooms and corridors' -- that the sum can be decomposed into sums 
over blocks which are large, but separated, and so close to independence.
For example, writing $\ra{W}{\tau/2}_n = 
(\sum_{i=m+1}^{m+n} X_i)/\sqrt{n} + \ra{Z}{\tau/2}$, Theorem \ref{thm:contmix}
shows that 
$$\alpha( \ra{V}{\tau/2}_m , \ra{W}{\tau/2}_n) \leq
\alpha( \ra{V}{\tau/2}_{m-\sqrt{m}} , \ra{W}{\tau/2}_n) + f_K(1/\sqrt{m}) 
= \alpha(\sqrt{m}) + f_k(1/\sqrt{m}).$$
In the notation of 
Theorem \ref{thm:fishsub}, $c(m) = C(K, \tau/2, \epsilon)
(\alpha(\sqrt{m}) + f_k(1/\sqrt{m}))^{1/3-\epsilon}$.

We first establish convergence along the `powers of 2 subsequence'
$S_k = \rat{V}_{2^k}$, writing $\wt{S}_k$ for $(\sum_{i=2^k}^{2^{k+1}} 
X_i + \rat{Z}_i)/\sqrt{2^k}$, since
$$ \fish(S_{k+1}) \leq \fish(S_k) + c(k) - \Delta(S_k, \wt{S}_k, 
1/2)$$ where $c(k) \rightarrow 0$. Then use an argument structured like
Linnik's proof \cite{linnik}.
Given $\epsilon$, we can find $K$ such that $c(k) \leq \epsilon/2$, for all
$k \geq K$. Now 
\begin{enumerate}
\item{either for all $k \geq K$, $2c(k) \leq 
\Delta(S_k,\wt{S}_k, 1/2)$, and so
$$ \fish(S_{k}) - \fish(S_{k+1}) \geq \Delta(S_k, \wt{S}_k, 1/2)/2,$$ 
so summing the telescoping sum, we deduce that
$ \sum_k \Delta(S_k, \wt{S}_k, 1/2)$ is finite, 
and hence there exists $L$ such 
that $\Delta(S_L, \wt{S}_L, 1/2) \leq \epsilon$.}
\item{or  for some $L \geq K$, $2c(L) \geq 
\Delta(S_L, \wt{S}_L, 1/2)$, then $\Delta(S_L, \wt{S}_L, 1/2) \leq \epsilon$.}
\end{enumerate}
Thus, in either case, there exists $L$ such that 
$\Delta(S_L , \wt{S}_L, 1/2) \leq \epsilon$, and hence by Proposition 
\ref{prop:deltadom}, $\fish(S_L) \leq \nu(\epsilon).$

Now, for any $k \geq L$, either $\fish(S_{k+1}) \leq \fish(S_k)$, or
$\Delta(S_k, \wt{S}_k, 1/2) \leq c(k) \leq \epsilon$. In the 
second case, $\fish(S_k) \leq \nu(\epsilon)$, so that
$\fish(S_{k+1}) \leq \nu(\epsilon) + \epsilon$. In either case, we prove
by induction that for all $k \geq L$, that
$\fish(S_{k+1}) \leq \nu(\epsilon) + \epsilon$.

We can fill in the gaps to gain control of the whole sequence, adapting
the proof of the standard sub-additive inequality, using the methods 
described in Appendix 2 of Grimmett \cite{grimmett2}.
\end{proof}

\section{Proof of sub-additive relations} \label{sec:mainproof}

This is the key part of the argument, proving the  bounds at the heart of
the limit theorems. However, although the analysis is somewhat involved,
it is not technically difficult.

We introduce notation where it will be clear whether densities and score
functions are associated with joint or marginal distributions, by their
number of arguments: $\rho_X(x)$ will be the score function of $X$, 
and $p'_X(x)$ the derivative of its density. For joint densities
$p_{X,Y}(x,y)$, $\der{p}{1}{X,Y}(x,y)$ will be the derivative of the 
density with respect to the first argument and $\der{\rho}{1}{X,Y}(x,y)
= \der{p}{1}{X,Y}(x,y)/p_{X,Y}(x,y)$, and so on.

Note that a similar equation to the independent case tells us about the 
behaviour of Fisher Information of sums: 
\begin{lemma}
If $X$, $Y$ are random variables, 
with joint density $p(x,y)$, and score functions $\der{\rho}{1}{X,Y}$ and 
$\der{\rho}{2}{X,Y}$ then $X+Y$ has score function $\widetilde{\rho}$
given by
$$ \widetilde{\rho}(z) = 
    \ep \left[ \left. \der{\rho}{1}{X,Y}(X,Y) \right| X+Y=z \right] =
    \ep \left[ \left. \der{\rho}{2}{X,Y}(X,Y) \right| X+Y=z \right]
.$$
\end{lemma}
\begin{proof}
Since $X+Y$ has density $ r(z) = \int p_{X,Y}(z-y,y) dy$, then: 
$$r'(z) = \int \der{p}{1}{X,Y}(z-y,y) dy.$$
Hence dividing, we obtain that:
$$ \widetilde{\rho}(z) = \frac{r'(z)}{r(z)}   = 
\int \der{\rho}{1}{X,Y}(z-y,y) \frac{p_{X,Y}(z-y,y)}{r(z)} dy,$$
as claimed. \end{proof}

For given $a,b$, define the function $M(x,y) = M_{a,b}(x,y)$ by:
$$ M(x,y) = a \left( \der{\rho}{1}{X,Y}(x,y)- \rho_X(x) \right) +  
b \left( \der{\rho}{2}{X,Y}(x,y)- \rho_Y(y) \right),$$ which is zero 
if $X$ and $Y$ are independent. Using properties of the perturbed density,
we will show that if $\alpha(S,T)$ is small, 
then $M$ is close to zero.
\begin{proposition} \label{prop:fishdecomp}
If $X,Y$ are random variables, with marginal score functions $\rho_X,\rho_Y$, 
and if the sum $\sqrt{\beta} X + \sqrt{1-\beta} Y$ has score function 
$\widetilde{\rho}$ then 
\begin{eqnarray*}
\lefteqn{ \beta J(X) + (1-\beta) J(Y) - J \left(\sqrt{\beta} X + 
\sqrt{1-\beta} Y \right)} \\
& &  + 2 \sqrt{\beta(1-\beta)} \ep \rho_X(X) \rho_Y(Y)
+ 2 \ep M_{\sqrt{\beta},\sqrt{1-\beta}}(X,Y) 
\widetilde{\rho}(X+Y)  \\
& = & \ep \left( \sqrt{\beta} \rho_X(X) + \sqrt{1-\beta}
\rho_Y(Y) - \widetilde{\rho} \left(\sqrt{\beta} X + \sqrt{1-\beta} Y \right) 
\right)^2  \end{eqnarray*}
\end{proposition}
\begin{proof}
By the two-dimensional version of Stein's equation, for any function $f(x,y)$
and for $i=1,2$:
$$ \ep \der{\rho}{i}{X,Y}(X,Y) f(X,Y) = 
- \ep \ra{f}{i}(X,Y).$$
Hence, we know that taking $f(x,y) = \widetilde{\rho}(x+y)$, for any $a,b$:
$$ \ep (a \rho_X(X) + b \rho_Y(Y)) \widetilde{\rho}(X+Y) = 
(a+b) J(X+Y) - \ep M_{a,b}(X,Y) \widetilde{\rho}(X+Y).$$
By considering
$ \int p(x,y) \left( a\rho_X(x) + b\rho_Y(y) - (a+b) \widetilde{\rho}(x+y)
\right)^2 dx dy,$ dealing with the cross term with the expression above,
we deduce that:
\begin{eqnarray*} 
\lefteqn{a^2 J(X) + b^2 J(Y) - (a+b)^2 J (X + Y)}  \\
& & + 2 ab \ep \rho_X(X) \rho_Y(Y)
+ 2 (a+b) \ep M_{a,b}(X,Y) 
\widetilde{\rho}(X + Y)\\
& = & \ep \left( a\rho_X(X) + b \rho_Y(Y) - (a+b)
 \widetilde{\rho}(X + Y) \right)^2 \geq 0.
\end{eqnarray*}
As in the independent case, 
we can rescale, and consider $X' = \sqrt{\beta} X$, $Y' = \sqrt{1-\beta}
Y$, and take $a = \beta, b = 1-\beta$.  
Note that $\sqrt{\beta} \rho_{X'}(u) = \rho_X(u/\sqrt{\beta})$,
$\sqrt{1-\beta} \rho_{Y'}(v) = \rho_Y(v/\sqrt{1-\beta})$. 
\end{proof}

Next, we require an extension of Lemma 3 of Barron \cite{barron} applied to 
single and bivariate random variables:

\begin{lemma} \label{lem:scoretail}
For any $S,T$, define $(X,Y) = (S+
\rat{Z}_S, T+\rat{Z}_T)$ and define $\ra{p}{2\tau}$ for the density of
$(S+\ra{Z}{2\tau}_S,T+\ra{Z}{2\tau}_T)$. 
There exists a constant $c_{\tau,k} 
= \sqrt{2} (2k/\tau e)^{k/2}$ such that for all $x,y$:
\begin{eqnarray*} 
\rat{p}_X(x) |\rho_X(x)|^{k} & \leq & c_{\tau,k} \ra{p}{2\tau}(x) \\
\rat{p}(x,y) |\der{\rho}{1}{X,Y}(x,y)|^{k} & \leq & c_{\tau,k} 
\ra{p}{2\tau}(x,y) \\
\rat{p}(x,y) |\der{\rho}{2}{X,Y}(x,y)|^{k} & \leq & c_{\tau,k} 
\ra{p}{2\tau}(x,y) 
\end{eqnarray*}
and hence 
$$ \left( \ep | \rho_X(X) |^k \right)^{1/k}
\leq \sqrt{ \frac{2^{1/k} 2k}{\tau e} }.$$ 
\end{lemma}
\begin{proof} 
We adapt Barron's proof, using H\"{o}lder's inequality 
and the bound; $(u/\tau)^{k} \phi_{\tau}(u) \leq c_{\tau,k} \phi_{2\tau}(u)$ 
for all $u$. 
\begin{eqnarray*}
p'_X(x)^{k} & = & \left( \ep \left( \frac{x-S}{\tau} \right)
\phi_{\tau}(x-S) \right)^{k} \\
& \leq & \left( \ep \left( \frac{x-S}{\tau} \right)^{k} 
\phi_{\tau}(x-S) \right)
\left( \ep  \phi_{\tau}(x-S) \right)^{k-1} \\
& \leq & c_{\tau,k} \left( \ep \phi_{2\tau}(x-S) \right) p_X(x)^{k-1}
\end{eqnarray*}
A similar argument gives the other bounds.
\end{proof}

Now, the normal perturbation ensures that the density doesn't decrease too
large, and so the modulus of the score function can't grow too fast.

\begin{lemma} \label{lem:scorel2}
Consider $X$ of the form 
$X = S + \rat{Z}_S$, where $\var S \leq K \tau$. If $X$ has
score function $\rho$, then for $B> 1$:
$$ \int_{-B\sqrt{\tau}}^{B\sqrt{\tau}} \rho(u)^2 du \leq 
\frac{8B^3}{\sqrt{\tau}} \left( 3+ 2K \right).$$
\end{lemma}
\begin{proof} As in Proposition \ref{prop:deltadom},
$p(u) \geq (2\exp 2K)^{-1} \phi_{\tau/2}(u)$, so that
for $u \in (-B\sqrt{\tau}, B\sqrt{\tau})$,
$(B\sqrt{\tau} p(u))^{-1} \leq 2\sqrt{\pi} \exp(B^2 + 2K)/B
\leq 2\sqrt{\pi} \exp(B^2 + 2K)
$. Hence for any $k \geq 1$, by H\"older's inequality:
\begin{eqnarray*}
\int_{-B\sqrt{\tau}}^{B\sqrt{\tau}} \rho(u)^2 du & \leq & 
\left( \int_{-B\sqrt{\tau}}^{B\sqrt{\tau}} |\rho(u)|^{2k} du 
\right)^{1/k} \left( 2B \sqrt{\tau} \right)^{1-1/k} \\
& \leq & \left( \int_{-B\sqrt{\tau}}^{B\sqrt{\tau}} 
\frac{ p(u) |\rho(u)|^{2k}}{ 2B \sqrt{\tau} \inf_u p(u)} du 
\right)^{1/k} \left( 2B \sqrt{\tau} \right) \\
& \leq &  \left( \frac{8B}{\sqrt{\tau}} \right) k \left( 2\sqrt{2\pi} 
\exp(B^2 + 2K) \right)^{1/k} \exp(-1).  
\end{eqnarray*}
Since we have a free choice of $k \geq 1$ to maximise $k \exp(v/k)$,
choosing $k=v \geq 1$  means that $k \exp (v/k) \exp(-1) = v
$. Hence we obtain a bound of
$$ \int_{-B\sqrt{\tau}}^{B\sqrt{\tau}} \rho(u)^2 du \leq  
\frac{8B}{\sqrt{\tau}} \left( B^2 + 2K + \log(2\sqrt{2\pi}) 
\right) 
\leq \frac{8B^3}{\sqrt{\tau}} \left( 3+ 2K \right).$$
\end{proof}

By considering $S$ normal, so that $\rho$ grows linearly with $u$, we know
that the $B^3$ rate of growth is a sharp bound.

\begin{lemma} \label{lem:fkgdensbound}
For random variables $S,T$, let  
$X = S + \rat{Z}_S$ and $Y = Y + \rat{Z}_T$, define
$L_B = \{ |x| \leq B \sqrt{\tau}, |y| \leq B \sqrt{\tau} \}$.
If $\max (\var S, \var T) \leq K \tau$ then
there exists a function $f_1(K,\tau)$ such that for 
$B \geq 1$:
$$ \ep M_{a,b}(X,Y) \widetilde{\rho}(X+Y) \1( (X,Y) \in L_B) 
\leq  \alpha(S,T) B^4 (a+b) f_1(K,\tau).$$
\end{lemma}
\begin{proof}
Lemma 1.2 of Ibragimov \cite{ibragimov} states that if $\xi, \nu$ are random
variables measurable with respect to ${{\cal{A}}}, {{\cal{B}}}$ respectively,
with $|\xi| \leq C_1$ and $|\nu| \leq C_2$ then:
$$ | \cov( \xi, \nu) | 
\leq 4C_1 C_2 \alpha({{\cal{A}}}, {{\cal{B}}}).$$ 

Now since $|\phi_\tau(u)| \leq 1/\sqrt{2 \pi \tau}$, and 
$|u \phi_\tau(u)/\tau| \leq \exp(-1/2)/\sqrt{2 \pi \tau^2}$, we deduce that:
$$ | p_{X,Y}(x,y) - p_X(x) p_Y(y) |
=  | \cov( \phi_{\tau}(x-S), \phi_{\tau}(y-T)) | 
\leq   \frac{2}{\pi \tau} \alpha(S,T). $$
Similarly:
\begin{eqnarray*}
|\der{p}{1}{X,Y}(x,y) - p_X'(x) p_Y(y) | 
& = & \left| \cov \left( \left( \frac{x-S}{\tau} \right)
 \phi_{\tau}(x-S) , \phi_{\tau}(y-T) \right) \right| \\
& \leq & 4 \left( \frac{\exp(-1/2)}{\sqrt{2 \pi \tau^2}} 
\frac{1}{\sqrt{2 \pi \tau}} \right) \alpha(S,T).
\end{eqnarray*}
By rearranging $M_{a,b}$, we obtain:
$$ p_{X,Y}(x,y) |M_{a,b}(x,y)| \leq \frac{2\alpha(S,T)}{\pi \tau}
\left( \frac{a+b}{\sqrt{\tau e}} + |a \rho_X(x) + b \rho_Y(y)| \right).$$
By Cauchy-Schwarz:
\begin{eqnarray*} 
\lefteqn{ \int p_{X,Y}(x,y) M_{a,b}(x,y) \widetilde{\rho}(x+y)
\1( (x,y) \in L_B) dx dy } \\
& \leq & \left( \frac{2 \alpha(S,T)}{\pi \tau} \right)
 \sqrt{ 32 B^4 (3 + 2K) } (a+b)
\left( \frac{\sqrt{4 B^2 \tau }}{\sqrt{\tau e}} 
+ \sqrt{ 16B^4(3+2K)} \right) \\
& \leq &  
\alpha(S,T) B^4 (a+b) \left( \frac{40\sqrt{2} (3+2K)}{\tau} \right).
\end{eqnarray*}
This follows firstly since by Lemma \ref{lem:scorel2}
$$ \int \rho_X(x)^2 \1((x,y) \in L_B) dx dy  \leq  (2B\sqrt{\tau}) 
\int_{-B\sqrt{\tau}}^{B\sqrt{\tau}} \rho_X(x)^2 dx 
\leq 16B^4 (3+2K)$$
and by Lemma \ref{lem:scorel2}
\begin{eqnarray*}
\lefteqn{\int \widetilde{\rho}(x+y)^2 \1((x,y) \in L_B) dx dy} \\
 & \leq & \int \widetilde{\rho}(x+y)^2 \1(|x+y| \leq 2B\sqrt{\tau})  
\1(|y| \leq B\sqrt{\tau})  dx dy \\
& \leq & 2B \sqrt{\tau} 
\int_{-2B\sqrt{\tau}}^{2B\sqrt{\tau}} 
\widetilde{\rho}(z)^2 dz \leq  32 B^4 (3 + 2K).
\end{eqnarray*}
\end{proof}
Now uniform decay of the tails gives us control everywhere else:
\begin{lemma} \label{lem:scoretailoff}
For $S,T$ with mean zero and variance 
$\leq K\tau$, let  
$X = S + \rat{Z}_S$ and $Y = T + \rat{Z}_T$.
There exists a function $f_2(\tau,K,\epsilon)$ such that:
$$ \ep M_{a,b}(X,Y) \widetilde{\rho}(X+Y) 
 \1( (X,Y) \notin L_B) dx dy \leq (a+b) 
\frac{f_2(\tau,K,\epsilon)}{B^{2-\epsilon}}.$$
For $S,T$ with $k$th moment ($k \geq 2$) bounded above, we can achieve a 
rate of decay of $1/B^{k-\epsilon}$.
\end{lemma}
\begin{proof} By Chebyshev
$\pr \left( (S+\ra{Z}{2\tau}_S,  T+\ra{Z}{2\tau}_T ) \notin L_B) \right)
\leq \int \ra{p}{2\tau}(x,y) (x^2 + y^2)/(2B^2 \tau) dx dy \leq (K+2)/B^2$
so by H\"older-Minkowski for $1/p + 1/q =1$:
\begin{eqnarray*}
\lefteqn{ \ep \der{\rho}{1}{X,Y}(X,Y) \widetilde{\rho}(X+Y) 
\1( (X,Y) \notin L_B)}  \\
& \leq &  \left( \ep | \der{\rho}{1}{X,Y}(X,Y)|^p \1((X,Y) \notin L_B)
\right)^{1/p}
\left( \ep | \widetilde{\rho}(X+Y) |^q \right)^{1/q} \\
& \leq & c_{\tau,p}^{1/p} c_{\tau,q}^{1/q}
\pr \left( (S+\ra{Z}{2\tau}_S,  T+\ra{Z}{2\tau}_T ) \notin L_B) \right)^{1/p} 
\\
& \leq & \frac{ 2\sqrt{2} \exp(-1)}{\tau} \sqrt{pq} \left( 
\frac{K+2}{B^2} \right)^{1/p}
\end{eqnarray*}
By choosing $p$ arbitrarily close to 1, we can obtain the required expression.
The other terms work in a similar way.
\end{proof}
Similarly we bound the remaining product term:
\begin{lemma} \label{lem:prodbound}
For random variables $S,T$ with mean zero and variances satisfying 
$\max(\var S, \var T) \leq K\tau$, let  
$X = S + \rat{Z}_S$ and $Y = T + \rat{Z}_T$.
There exist functions $f_3(\tau,K)$ and $f_4(\tau,K)$ such that
$$\ep \rho_X(X) \rho_Y(Y)
\leq  f_3(\tau,K) B^4 \alpha(S,T) + f_4(\tau,K) /B^{2}. $$
\end{lemma}
\begin{proof}
Using part of Lemma \ref{lem:fkgdensbound}, we know that 
$p_{X,Y}(x,y) - p_X(x) p_Y(y) \leq 2\alpha(S,T)/(\pi \tau)$. Hence by an
argument similar to that of Lemmas \ref{lem:scoretailoff}, we obtain that:
\begin{eqnarray*}
\ep \rho_X(X) \rho_Y(Y) 
& = & \int \left( p_{X,Y}(x,y) - p_X(x) p_Y(y) \right)
\rho_X(x) \rho_Y(y) dx dy \\
& \leq & \frac{2\alpha(S,T)}{\pi \tau}
\int  |\rho_X(x)| |\rho_Y(y)| \1((x,y) \in L_B) dx dy \\
& & + \int p(x,y) |\rho_X(x)| |\rho_Y(y)| \1((x,y) \notin L_B) dx dy \\
& & + \int p(x) p(y) |\rho_X(x)| |\rho_Y(y)| \1((x,y) \notin L_B) dx dy  \\
& \leq & \frac{2\alpha(S,T)}{\pi \tau} \left(
\int_{-B\sqrt{\tau}}^{B\sqrt{\tau}} |\rho_X(x)|^2 dx \right)^2 \\
& & + 2 \left( \int p_{X,Y}(x,y) |\rho_X(x)|^2 \1((x,y) \notin L_B) dx dy
\right). \\
\end{eqnarray*}
as required.
\end{proof}

\begin{proof}{\bf of Theorem \ref{thm:fishsub}}
Combining Lemmas  \ref{lem:fkgdensbound}, \ref{lem:scoretailoff} and
\ref{lem:prodbound} , we obtain for given $K, \tau, \epsilon$  that
there exist constants $C_1, C_2$ such that
$$\ep M_{\sqrt{\beta},\sqrt{1-\beta}} \widetilde{\rho} 
+ \sqrt{\beta(1-\beta)} \ep \rho_X \rho_Y \leq
C_1 \alpha(S,T) B^4 +  C_2/B^{2-\epsilon},$$ so choosing $B = 
(1/4\alpha(S,T))^{1/6} > 1$, we obtain a bound of 
$C \alpha(S,T)^{1/3-\epsilon}$.

By Lemma \ref{lem:scoretailoff}, 
note that if $X,Y$ have bounded $k$th moment, then we obtain decay
at the rate
$C_1 \alpha(S,T) B^4 +  C_2/B^{k'}$, for any $k' < k$.
Choosing $B = \alpha(S,T)^{-1/(k'+4)}$, we obtain a rate of 
$\alpha(S,T)^{k'/(k'+4)}$. 
\end{proof}

\section{ Control of the mixing coefficients} \label{sec:contmix}
To control $\alpha(X+Z,Y)$ and to prove
Theorem \ref{thm:contmix}, we use truncation, smoothing and triangle 
inequality arguments similar to those of the previous section. 
Write
$W$ for $X+Z$, $L_B = \{ (x,y): |x| \leq B \sqrt{\tau}, |y| \leq B
\sqrt{\tau} \}$, and $\overline{R}$ for $R \cap (-B\sqrt{\tau}, B\sqrt{\tau})$.
Note that by Chebyshev, $\pr((W,Y) \in L_B^c) \leq  \pr(|W| \geq B\sqrt{\tau})
+  \pr(|Y| \geq B\sqrt{\tau}) \leq 2(K+1)/B^2$.
 Hence by the triangle inequality, for any
sets $S,T$:
\begin{eqnarray*}
\lefteqn{|\pr( (W,Y) \in (S,T)) - \pr( W \in S) \pr( Y \in T) | } \\
& \leq & |\pr( (W,Y) \in (S,T) \cap L_B) - \pr( W \in \overline{S}) 
\pr( Y \in \overline{T}) | \\
& &
+ \pr( (W,Y) \in L_B^c) + \pr( |W| \geq B\sqrt{\tau}) \pr(|Y| \geq
B\sqrt{\tau}) \\
& \leq & |\pr( (W,Y) \in (\overline{S},\overline{T}) ) 
- \pr( (X,Y) \in (\overline{S},\overline{T}) )  \\
& & + |\pr( (X,Y) \in (\overline{S},\overline{T}) )  - 
\pr( X \in \overline{S}) \pr( Y \in \overline{T}) | \\
& & + | \pr( X \in \overline{S}) 
- \pr( W \in \overline{S}) |
\pr( Y \in \overline{T}) + \frac{4(K+1)}{B^2} \\
& \leq & \int | p_{W,Y}(w,y) - p_{X,Y}(w,y) | \1((w,y) \in L_B)dw dy
+ \alpha(X,Y) \\
& & + \int | p_{X}(w)  - p_W(w) | \1(|w| \leq B\sqrt{\tau}) dw  + 
\frac{4(K+1)}{B^2} \\
\end{eqnarray*}
Here, the first inequality follows on splitting $\re^2$ into $L_B$ and 
$L_B^c$, the second by repeated application of the triangle inequality,
and the third by expanding out probabilities using the densities.
Now the key result is that:
\begin{proposition}
For $S$ and $T$, define  $X = S+ \rat{Z}_S$ and  $Y = T+ \rat{Z}_T$,
where $\max(\var S, \var T) \leq K\tau$. 
If $Z$ has variance
$\epsilon$, then there exists a constant $C = C(B, K, \tau)$ such that:
$$ \int |p_W(w) - p_X(w)| \1(|w| \leq B\sqrt{\tau}) dw \leq
(\exp(C \epsilon^{1/5}) - 1) + 2 \epsilon^{1/5}.$$
\end{proposition}
\begin{proof}
We can show that for $|z| \leq \delta^2$ and $|x| \leq B\sqrt{\tau}$:
\begin{eqnarray*}
\frac{ p_{X,Z}(x-z,z)}{ p_{X,Z}(x,z)} 
& = & \exp \left( \int_{x-z}^x \der{\rho}{1}{X,Z}(u,z) du \right) \\
& \leq & \exp \left( \left( \int_{-2B\sqrt{\tau}}^{2B\sqrt{\tau}} 
\der{\rho}{1}{X,Z}(u,z)^2 du \right)^{1/2} \delta \right) \\
& \leq & \exp C \delta, 
\end{eqnarray*}
by adapting Lemma \ref{lem:scorel2} to cover bivariate random variables.
Hence we know that:
\begin{eqnarray*}
\lefteqn{\int |p_W(w) - p_X(w)| \1(|w| \leq B\sqrt{\tau}) dw} \\ 
& \leq & \int |p_{X,Z}(w-z,z) - p_{X,Z}(w,z)| \1(|z| \leq \delta^2,
|w| \leq B\sqrt{\tau}  ) dz dw \\
& & + \int |p_{X,Z}(w-z,z) - p_{X,Z}(w,z)| \1(|z| \geq \delta^2) dw dz \\
& \leq & \int p_{X,Z}(w,z) (\exp C \delta -1 ) dw dz
+ 2 \pr (|Z| \geq \delta^2)  \\
& \leq & (\exp C \delta  -1 )
+ 2 \pr (|Z| \geq \delta^2)     
\end{eqnarray*}
Thus choosing $\delta = \epsilon^{1/5}$, the result follows.
\end{proof}

Similar analysis allows us to control
$$\int | p_{W,Y}(w,y) - p_{X,Y}(w,y) | \1((w,y) \in L_B)dw dy.$$

\section*{Acknowledgements}
The author is a Fellow of Christ's College Cambridge, who funded
travel and research expenses. I also thank Yuri Suhov of the Statistical 
Laboratory for useful discussions, and the anonymous referee for helpful 
comments.

\end{document}